\documentclass[11pt,reqno]{amsproc}
\usepackage[margin=1in]{geometry}
\usepackage{amsmath, amsthm, amssymb}
\usepackage{times, esint,stackrel}
\usepackage{enumitem}
\usepackage{color}
\usepackage{enumitem}   
\usepackage{hyperref}

\usepackage{color}
\usepackage{hyperref}

\newcommand{\be}{\begin{equation}}
\newcommand{\ee}{\end{equation}}
\newcommand{\bea}{\begin{eqnarray}}
\newcommand{\eea}{\end{eqnarray}}

\newtheorem{thm}{Theorem}
\newtheorem{prop}{Proposition}
\newtheorem{lemma}{Lemma}

\theoremstyle{definition}
\newtheorem{rem}{Remark}

\renewcommand{\div}{{\mbox{div}\,}}

\newcommand{\ve}{{\varepsilon}}

\newcommand{\red}[1]{\textcolor{black}{#1}}

\def\wc{\rightharpoonup}

\def\curl{\text{curl}}
\def\div{\text{div}}

\newcommand{\bq}{\begin{equation}}
\newcommand{\eq}{\end{equation}}
\newcommand{\bqa}{\begin{eqnarray*}}
\newcommand{\eqa}{\end{eqnarray*}}
\newcommand{\Rr}{\mathbb{R}}

\newcommand{\om}{\omega}
\newcommand{\p}{\partial}

\title[On the Emergence of Weak Euler Solutions]{Remarks on the emergence of weak Euler \\ solutions  in the vanishing viscosity limit}
\author{Theodore D. Drivas}
\address{Department of Mathematics, Princeton University, Princeton, NJ 08544}
\email{tdrivas@math.princeton.edu}
\author{Huy Q. Nguyen}
\address{Department of Mathematics, Princeton University, Princeton, NJ 08544}
\email{qn@math.princeton.edu}

\date{today}
\begin{document}

\begin{abstract}
We prove that if the \emph{local} second-order structure function exponents in the inertial range remain positive uniformly in viscosity, then any spacetime $L^2$  weak limit of Leray--Hopf weak solutions of the Navier-Stokes equations on any bounded domain $\Omega\subset \mathbb{R}^d$, $d= 2,3$  is a weak solution of the Euler equations. This holds for both no-slip and Navier-friction conditions with viscosity-dependent slip length. The result allows for the emergence of non-unique, possibly dissipative, limiting weak solutions of the Euler equations.  
\end{abstract}

\maketitle

 \section{Introduction}
The incompressible Navier-Stokes equations governing viscous flow contained in $\Omega\subset \mathbb{R}^d$, $d\geq 2$ read
\begin{align}\label{NSb}
\partial_t u^\nu + u^\nu \cdot \nabla u^\nu  = - \nabla p^\nu  + \nu \Delta u^\nu  + f^\nu   \qquad &\text{in} \quad   \Omega, \\
\nabla \cdot u^\nu  = 0    \qquad\qquad\qquad \quad\qquad\ \  &\text{in} \quad  \Omega \times (0,T),\\
u^\nu\vert_{t=0} = u_0  \qquad\qquad\qquad \quad\qquad  &\text{in}  \quad   \Omega. \label{NSf}
\end{align}
These must be supplied with conditions on the boundary.   The most common choice is no-slip conditions
\begin{align} \label{noslip}
u^\nu = 0 \quad  \text{on}  \quad  \partial \Omega \times (0,T).
\end{align}
No-slip, or stick, boundary conditions assert that the fluid velocity matches the velocity of the boundary (which we here consider stationary).  Another possible choice is that of  Navier-friction or slip boundary conditions
\begin{align} \label{NSslip1}
2(D(u^\nu)  \hat{n})_\tau&= -\alpha^\nu u^\nu_\tau\quad \text{on}  \quad  \partial \Omega \times (0,T),\\
u^\nu \cdot \hat{n} &= 0 \quad\quad\quad \ \ \text{on}\quad  \partial \Omega \times (0,T), \label{NSslip2}
\end{align}
 where $\hat{n}:=\hat{n}(x)$ is the vector field normal to the boundary $\partial \Omega$, $D(v):= 1/2(\nabla v +( \nabla v)^t)$ is the strain tensor, $f_\tau$ denotes the tangential-to-boundary part of $f$, and   the inverse slip-length $\alpha^\nu:= \alpha^\nu(x)$ is a $C^2$ positive function uniformly--in--$\nu$ bounded from below.   The conditions \eqref{NSslip1} form $d-1$ constraints on the solution. As the name suggests, this condition allows the fluid to slip tangentially along the boundary for all $\nu>0$. The inverse slip length is often assumed to be of power-law form $\alpha^\nu=c_0 \nu^{-\beta}$ for $\beta \in[0, 1]$ and $c_0>0$.  The critical and physical setting has $\alpha^\nu =c_0\nu^{-1}$, which arises rigorously from kinetic theory considerations \cite{JM17}.

The Navier-Stokes equations, coupled with either \eqref{noslip} or \eqref{NSslip1}--\eqref{NSslip2} boundary conditions, are known to admit global weak solutions that obey an energy inequality \cite{L34}. Specifically, we denote 
 \be
H(\Omega):= \{w\in L^2(\Omega), \ {\rm div} \ w = 0, \ w\cdot \hat{n}|_{\partial\Omega} = 0 \}.
 \ee
 Then, for any $u_0^\nu\in H(\Omega)$, there exists a Leray-Hopf weak solution $u^\nu \in L_{loc}^\infty(0,\infty;L^2(\Omega))\cap L_{loc}^2(0,\infty;H^1(\Omega))$ of \eqref{NSb}--\eqref{NSf} satisfying  no-slip boundary conditions  \eqref{noslip}   in the sense that
 \be\label{weakNS}
 \int_0^T\int_\Omega  \big(u^\nu \cdot  \partial_t\varphi   +  u^\nu \otimes u^\nu  :  \nabla  \varphi  -\nu \nabla  u^\nu: \nabla \varphi \big)dxdt =0
\ee
for all solenoidal test vector fields $\varphi \in C_c^\infty((0,\infty)\times \Omega)$.  Such solutions obey the global energy inequality
\be\label{enD}
\frac{1}{2} \int_\Omega |u^\nu(t)|^2 dx\leq \frac{1}{2} \int_\Omega |u_0^\nu|^2 dx -\nu \int_0^t \int_\Omega |\nabla u^\nu(s)|^2 dx ds
\ee
for all $t\geq 0$.  In the case of Navier-friction boundary conditions \eqref{NSslip1}--\eqref{NSslip2}, there exists 
 global Leray-Hopf weak solutions $u^\nu$ (see Theorem 2 of \cite{IS11}) satisfying
 \be\label{weakNSfric}
 \int_0^T\int_\Omega  \big(u^\nu \cdot  \partial_t\varphi   +  u^\nu \otimes u^\nu  :  \nabla  \varphi  -2\nu D(  u^\nu): D( \varphi )  \big)dxdt =  \nu  \int_0^T\int_{\partial\Omega}  \alpha^\nu u^\nu \cdot \varphi dS
\ee
for all solenoidal  $\varphi \in C_c^\infty((0,\infty)\times \overline{\Omega})$ and  the global energy inequality for all $t\geq 0$
\begin{align}
\frac{1}{2} \int_\Omega |u^\nu(t)|^2 dx&\leq \frac{1}{2} \int_\Omega |u_0^\nu|^2 dx -2\nu \int_{0}^t \int_\Omega |D(  u^\nu(s))|^2 dx ds   - \nu    \int_{0}^t \int_{\partial \Omega}\alpha^\nu|u^\nu|^2  dSds .    \label{enN}
\end{align}

The nature of the inviscid limit $\nu\to 0$ is an important issue in the study of fluid dynamics. On domains without boundary such as $\mathbb{T}^d$ or $\mathbb{R}^d$, if a strong Euler solution  $u^E$ exists for the same initial data, then it is well known that any sequence of Leray solutions of Navier-Stokes converges to that Euler solution in the energy norm,  i.e. strong convergence of $u^\nu$ to $u^E$ in $L_t^\infty L^2_x$ \cite{K72,P86}.
 On domains with boundary, even when a strong Euler solution exists in the background, this question is widely open.  Most of the known results establish the strong inviscid limit under a variety of conditions, see e.g. ~\cite{kato,BardosTiti13, ceiv,ConstantinKukavicaVicol15,Kelliher07,TemamWang97b,Kelliher06}.    For no-slip conditions, some unconditional results are known, but only for short times and under special simplifying assumptions such as analytic data~\cite{SammartinoCaflisch98b}, vanishing of the initial vorticity in a neighborhood of the boundary~\cite{Maekawa14}, \red{near shear \cite{GV16},} or special symmetries on the flow~\cite{Kelliher09,LopesMazzucatoLopesTaylor08,MazzucatoTaylor08}. These unconditional results hold for laminar flows, before any boundary layer separation or other characteristic turbulent behavior can occur.  

For turbulent flows with and without boundaries, the connection between Euler and the inviscid limit of Navier-Stokes remains a fundamental open problem. On domains with boundary, the lack of information on convergence even in the presence of a strong solution of Euler for arbitrary finite times motivates the need to establish the inviscid limit to some weaker notion of Euler solution.  One such notion, guaranteed to exist along suitable subsequences $\nu_n\to 0$, is that of measure valued solutions of Euler \cite{DM87}.
A  stronger and physically desirable relaxation is that of weak or distributional solutions. We say that \emph{weak inviscid limit} holds if $u^\nu$ converges (along a subsequence) weakly in $L_{t}^2L_x^2$ to a weak solution $u$ of the Euler equations as $\nu\to 0$. Without boundaries, \cite{Glimm,Isett17,DE17} show that any suitable fractional degree of regularity implies the weak inviscid limit. On domains with boundary, the situation is more subtle and less is understood.  The work of \cite{Filho05} establishes the weak inviscid limit for Navier-friction boundary conditions with constant slip-length and ``rough'' initial data.   Constantin \& Vicol \cite{CV17} later proposed sufficient conditions under which weak inviscid limit holds more generally for no-slip boundary condition.
 In 2D, they assume uniform local enstrophy bounds.
In 3D, they assume weak $L_x^2$ convergence of the velocity for each time-slice 
and a uniform bound on the local second-order structure function in the inertial range (down to an arbitrary scale $\eta(\nu)$  converging to $0$ as $\nu\to 0^+$).  We remark that, due to failure of weak-strong uniqueness on domains with boundaries \cite{W17}, convergence in the inviscid limit can hold to (possibly non-unique) weak solutions which may coexist with a strong Euler solution.  This is particularly relevant in 2D. Our main result is
\begin{thm}\label{theorem}
Let $\Omega\subset \mathbb{R}^d$, $d= 2,3$ be a bounded domain with $C^2$ boundary.
Let $\{u^{\nu_n}\}$ be a sequence of Leray-Hopf weak solutions on $\Omega\times [0, T]$ of \eqref{NSb}--\eqref{NSf} with either no-slip  \eqref{noslip} or Navier-friction boundary conditions \eqref{NSslip1}--\eqref{NSslip2}, viscosities $\nu_n\to 0$, initial data $u^{\nu_n}(0)$ uniformly bounded in $L^2(\Omega)$, and forces $f^{\nu_n}\wc f$ in $L^2(0, T; L^2(\Omega))$.
Assume that for every $U\Subset \Omega$, one can find a positive constant $C=C_U$ and a local structure-function scaling exponent $\zeta_2:=\zeta_2(U)\in (0,2)$ such that for all $n\ge 1$
\bq\label{StuctFunc}
S_2(r;U):= \int_0^T \int_U |u^{\nu_n} (x+r,t) - u^{\nu_n}(x,t)|^2 dx dt  \leq C |r|^{\zeta_2} \qquad   \forall \  \eta({\nu_n})\leq |r|< {\rm dist}(U,\partial\Omega),
\eq
where $\eta(\nu):= \nu^{1/(2-\zeta_2)}$. 
Then every weak limit $u$ of $u^{\nu_n}$ is $L^2(0, T; L^2(\Omega))$ is a weak solution of Euler with forcing $f$. Moreover, $u$ inherits the regularity $u\in L^2(0,T;B_2^{\zeta_2(U)/2,\infty}(U))$ for any $U\Subset \Omega$.
\end{thm}
Theorem \ref{theorem} removes assumption of weak convergence for each time-slice in the 3D result of \cite{CV17} and weakens their assumption of uniform local $L^2$ bounds for vorticity in the 2D case. In fact, we have

\begin{prop}\label{prop}
In any $C^2$ bounded domains $\Omega\subset \mathbb{R}^d$, $d=2, 3$ and with either no-slip or Navier-friction boundary conditions on the velocity $u^\nu$, the bound  \eqref{StuctFunc} is equivalent to following: for every $U\Subset \Omega$, there exists $\delta:=\delta(U)>0$ and a constant $C:=C_U$ such that for all $n\ge 1$
\be\label{vortbd}
\|\om^{\nu_n}\|_{L^2(0,T;H^{-1+\delta}(U))} \leq C, \qquad \om^{\nu_n}:= {\rm curl}(u^{\nu_n}).
\ee
\end{prop}
Coupling the proposition with Theorem \ref{theorem}, we see that uniform $L^2_tH_{loc}^{-1+}$ regularity on vorticity in either 2D or 3D is sufficient to conclude that the weak inviscid limit holds.

We observe in Lemma \ref{lemma1} below that the uniform assumption \eqref{StuctFunc} in the inertial range, together with uniform boundedness of viscous dissipation which follows from Leray's construction \eqref{enD} or \eqref{enN}, implies uniform Besov regularity.  The proof of this fact makes use of the notion of a dissipation scale, below which the Leray solution possesses a certain degree of regularity uniformly in viscosity.   Despite their equivalence, the practical advantage of stating the condition \eqref{StuctFunc} instead of uniform regularity is that the bound need be checked only within a finite  -- though ever increasing with $Re$ --  range of scales.   
Finally, we remark that the presence of solid boundaries adds severe obstructions to obtaining the strong inviscid limit.  On the other hand, our result shows that, for weak inviscid limit, we do not need to require any further conditions relative to the case without boundary in spite of the nonlocal nature of the equations.


 \section{Proof of Theorem \ref{theorem}}
 
For a given vanishing viscosity subsequence $\{\nu_n\}_{n\in \mathbb{N}}$, denote $u_n:=u^{\nu_n}$, $\omega_n:=\omega^{\nu_n}$ and $f_n=f^{\nu_n}$, \red{where $u_n$ are Leray-Hopf weak solutions of \eqref{NSb}-\eqref{NSf}}. Upon relabelling we have
\begin{align}
\label{wconv:u}
u_n\wc u \quad&\text{in}\ ~L^2(0, T; H(\Omega)),\\
\label{wconv:om}
\om_n\wc \om\quad&\text{in}\ ~L^2(0, T; H^{-1}(\Omega)),
\end{align}
and 
\bq
f_n\wc f\quad\text{in}~L^2(0, T; L^2(\Omega)).
\eq
It follows that $\omega=\curl (u)$ in the sense of distributions. We need to show that $u$ is a weak Euler solution.

\begin{lemma}\label{lemma1}
In any dimension, local structure function scaling in the inertial range \eqref{StuctFunc}  implies that Leray solutions $u_n$ are uniformly bounded in $L^2(0,T;B_2^{\zeta_2(U)/2,\infty}(U))$ for every $U\Subset \Omega$.  Recall that the Besov space $B_p^{\sigma,\infty}(U)$ is made up of measurable functions $f:U\subset \Rr^d\to \mathbb{R}$ which are finite in the norm
\be
\| f\|_{B_p^{\sigma,\infty}(U)}:= \| f\|_{L^p(U)} 
+ \sup_{r\in \Rr^d}\frac{\|f(\cdot+r) - f(\cdot)\|_{L^p(U\cap (U-\{r\}))}}{|r|^\sigma}
\label{besov-def} 
\ee
for $p\geq 1$ and $\sigma\in (0,1)$.
\end{lemma}
\begin{proof}
By energy balance for no-slip boundary conditions \eqref{enD}, the viscous dissipation is bounded by
\be\label{disscont}
\nu \int_0^T \|\nabla u^\nu(t)\|_{L^2(\Omega)}^2 dt \leq \frac{1}{2} \|u_0^\nu \|_{L^2(\Omega)}^2<{\rm (const.)}.
\ee
For Navier-friction conditions, a simple calculation gives
\begin{align}\nonumber
\nu \int_{0}^T \int_\Omega |\nabla u^\nu(t)|^2 dxdt &=2 \nu \int_{0}^T \int_\Omega |D (u^\nu(t))|^2 dx dt  -  \nu \int_{0}^T \int_{\partial\Omega}  \hat{n} \cdot (u^\nu \cdot \nabla) u^\nu dSdt\\
 &=2 \nu \int_{0}^T \int_\Omega |D (u^\nu(t))|^2 dx  dt +  \nu \sum_{i,j=1}^{d-1}   \int_{0}^T \int_{\partial \Omega} ( u^\nu \cdot \hat{\tau}_i) (\hat{\tau}_i \cdot \nabla \hat{n} \cdot \hat{\tau}_j) ( u^\nu \cdot \hat{\tau}_j) dS dt
\end{align}
where, for every $x\in \p\Omega$, $\{ \hat{\tau}_i(x)\}_{i=1}^{d-1}$ is an orthogonal basis of the tangent space of $\p\Omega$ at $x$. In passing from the first to second line, we used the fact that $u^\nu\cdot \nabla$ is a tangential-to-the-boundary derivative operator and the impermeability condition $\hat{n}\cdot u^\nu=0$.  Combined with the energy balance \eqref{enN}, we find
 \begin{align}\nonumber
\nu \int_{0}^T\int_\Omega |\nabla u^\nu(t)|^2 dxdt &\leq \frac{1}{2} \int_\Omega |u_0^\nu|^2 dx - \frac{1}{2} \int_\Omega |u^\nu(T)|^2 dx\\
& \!\!\! \!\!\! \!\!\! \!\!\! \!\!\! - \nu   \int_{0}^T  \int_{\partial \Omega}\alpha^\nu|u^\nu|^2  dS  dt+  \nu \sum_{i,j=1}^{d-1}  \int_{0}^T    \int_{\partial \Omega} ( u^\nu \cdot \hat{\tau}_i) (\hat{\tau}_i \cdot \nabla \hat{n} \cdot \hat{\tau}_j) ( u^\nu \cdot \hat{\tau}_j) dSdt. \label{lastterm}  
\end{align}
Since $\sqrt{\nu \alpha^\nu} u^\nu \in L^2(0,T;L^2(\partial\Omega))$ from the energy balance  \eqref{enN}, $\alpha^\nu$ uniformly-in-$\nu$ bounded from below, and $\hat{\tau}_i \cdot \nabla \hat{n} \cdot \hat{\tau}_j\in L^\infty(\partial\Omega)$ for all $i,j= 1, \dots, (d-1)$, the last term on the right-hand-side of \eqref{lastterm} is controlled by the initial kinetic energy.  Thus, we again obtain a bound on the dissipation of the form \eqref{disscont} \red{with the dependence constant $1/2$ replaced by some number depending only on $\|\hat\tau\cdot \nabla \hat n\cdot \hat\tau\|_{L^\infty(\p\Omega)}$}.

Then, we have the following bound on the second-order structure function
\be
S_2(r;U) \leq  |r|^2  \int_0^T \|\nabla u^{\nu_n}(t)\|_{L^2(\Omega)}^2 dt\leq C \left( \frac{|r|}{\sqrt{{\nu_n}}}\right)^2 \qquad \forall \   |r|< {\rm dist}(U,\partial\Omega),
\ee
for some constant $C$ independent of viscosity.  For all $r$ such that $|r|\leq \eta_n:= {\nu_n}^{1/(2-\zeta_2)}$, we have $(|r|/\sqrt{\nu_n})^2   \leq |r|^{\zeta_2}$. Combined with the  assumption \eqref{StuctFunc} we have over the whole range 
\be
S_2(r;U) \leq C  |r|^{\zeta_2}  \qquad  \forall \  |r|<{\rm dist}(U,\partial\Omega).
\ee
Since $u^\nu$ is uniformly bounded in $L^\infty(0,T;L^2(\Omega))$, this establishes Besov regularity as claimed.
\end{proof}

\begin{lemma}\label{lemma2}
\red{If $u_n$ are uniformly bounded in $L^2(0,T;B_2^{\zeta_2/2,\infty}(U))$ for some $\zeta_2=\zeta_2(U)\in (0,2)$ and every $U\Subset \Omega$, then} we have $\omega_n\to \omega$ strongly in $C(0,T;H_{loc}^{-1-\delta}(\Omega))$  for some $\delta>0$. 
\end{lemma}
\begin{proof}
Choosing test vector fields of the form $\curl \varphi$, $\varphi\in C_c^\infty(\Omega\times (0, T))$ in the weak formulations \eqref{weakNS} and \eqref{weakNSfric} we find that the vorticity $\om_n$ satisfies in the sense of spacetime distribution 
\be\label{vorteq}
\partial_t \omega_n + {\rm curl}( \div  (u_n\otimes u_n)) = \nu_n \Delta \omega_n+ g_n
\ee
where $g_n:={\rm curl}f_n$ uniformly bounded in $L^2(0, T; H^{-1}(\Omega))$.   Fix two open sets $V\Subset U\Subset \Omega$.  By Lemmas \ref{lemma1} and \ref{embed}, we have that  $u_n$ is uniformly bounded in  
$L^2(0,T;B_2^{\zeta_2(U)/2,\infty}(U))\subseteq L^2(0,T;L^p(V))$ with $ p\in (2,2d/(d-\zeta_2))$.  This combined with the uniform bound for the velocity in $L^\infty(0,T; L^2(V))$, gives that $u_n\otimes u_n\in L^2(0,T;L^{q}(V))$ for some $q\in (1,2]$. 
Moreover, it follows from the uniform boundedness of the viscous dissipation that the sequence $\sqrt{\nu_n} \  \omega_n$ is bounded in $L^2(0,T;L^2(\Omega))$.  Thus  \eqref{vorteq} implies 
\be
 \partial_t \omega_n \quad \text{is uniformly bounded in}\quad L^2(0,T;W^{-2,q}(V)).
\ee
Now for $\delta>0$  sufficiently small we have the embeddings $H^{-1}(U)\subset H^{-1-\delta }(V)\subset W^{-2, q}(V)$ where the first one is compact and the second one is continuous. Aubin-Lions's lemma then ensures that $\om_{n_k}\to w$ in $C(0, T; H^{-1-\delta}(V))$ for some subsequence $n_k\to \infty$ and some $w\in L^2(0, T; H^{-1}(U))$. A priori, both $n_k$ and $w$ depend on $U$. However, owing to the convergence \eqref{wconv:om}, we conclude that \red{$w=\omega$ and thus}
\bq\label{strong:om}
\om_n\to \omega\quad\text{in}~C(0, T; H^{-1-\delta}(V))
\eq
as a whole sequence.
\end{proof}
\begin{rem}
The proof does not make use of Lie-transport structure of the vorticity equation which holds for  strong solutions and distinguishes 2D and 3D. We simply take ${\rm curl}$ of the Navier-Stokes equation to eliminate the pressure but keep the nonlinearity as it is. Uniform local $L^p$ bounds on the velocity $u^\nu$ with $p>2$ are enough to conclude strong convergence \eqref{strong:om}.
\end{rem}
Next, from the convergence of vorticity for each time we deduce convergence of velocity:
\begin{lemma}\label{lemma3} 
\red{If  $\omega_n\to \omega$ strongly in $C(0,T;H_{loc}^{-1-\delta}(\Omega))$ for some $\delta>0$, then} the velocity converges weakly for almost every time, i.e.
\bq\label{u:timeslice}
u_n(t)\wc u(t)\quad\text{in}~L^2(\Omega)\quad\text{a.e.}~t\in [0, T].
\eq
\end{lemma}

\begin{proof}
 Fix $t_0\in [0, T]$. From the Leray energy inequality, there exist a subsequence $n_k$ and a function $v(t_0)$ both a priori depending on $t_0$ such that $u_{n_k}(t_0)\wc v(t_0)$ in $H(\Omega)$. Then \bq\label{u:timeslice1}
\om_{n_k}(t_0)={\rm curl}\ u_{n_k}(t_0)\wc{\rm curl}\  v(t_0)\quad\text{in}~H^{-1}(\Omega).
\eq
 On the other hand, the strong convergence \red{ $\omega_n\to \omega$ in $C(0,T;H_{loc}^{-1-\delta}(\Omega))$} implies that $\om_n(t_0)\to \om(t_0)$ in $H^{-1-\delta}_{loc}(\Omega)$. Then necessarily  ${\rm curl}\ v(t_0)=\om(t_0)={\rm curl}\ u(t_0)$ in $\Omega$. According to Corollary 2.9 of \cite{GR12}, 
 \be
(H(\Omega))^\perp:= \{w\in L^2(\Omega), \ {\rm curl} \ w = 0\}.
 \ee
 Since the difference $u(t_0)-v(t_0)\in H\cap H^\perp$, we  conclude that  $v(t_0)=u(t_0)$. Therefore, the whole sequence $u_n(t_0)$ converges weakly to $u(t_0)$ in $L^2(\Omega)$ for a.e. $t_0\in [0, T]$. 
\end{proof}
\begin{rem}
It is crucial in the above proof that a vorticity in $H^{-1}(\Omega)$ determines uniquely a velocity in $H(\Omega)$. Finiteness of kinetic energy is enough.
\end{rem}
Finally, Theorem 3.1 of Constantin and Vicol \cite{CV17} allows us to conclude that $u$ is a weak Euler solution.

\begin{lemma}[Theorem 3.1 of \cite{CV17}]\label{lemma4} 
If condition \eqref{StuctFunc} holds for \red{Leray-Hopf solutions} $u^{\nu_n}$ together with the weak convergence \eqref{u:timeslice} for a.e. time-slice, then $u$ is a weak solution of the Euler equations with force $f$.
\end{lemma}

To see that the limiting weak solution inherits some Besov regularity, note that along a subsequence $u_n(\cdot+r,\cdot) - u_n(\cdot,\cdot)\wc u (\cdot+r,\cdot) - u(\cdot,\cdot)$ in $L^2(U\times (0, T))$ for all $|r|< {\rm dist}(U,\partial\Omega)$.  Moreover, if $w_n\wc w$ in $L^2(O)$, then $\|w\|_{L^2(O)} \leq \liminf_{n\to \infty} \|w_n\|_{L^2(O)}$. Passing to the limit in the bound \eqref{StuctFunc} gives 
\bq\label{StuctFuncE}
\int_0^T \int_U |u (x+r,t) - u(x,t)|^2 dx dt  \leq C |r|^{\zeta_2} \qquad   \forall \   |r|< {\rm dist}(U,\partial\Omega),
\eq
and thus $u\in L^2(0,T;B_2^{\zeta_2(U)/2,\infty}(U))$. This concludes the proof of Theorem \ref{theorem}.

\begin{rem}
\red{According to Theorem 1.1 in \cite{LNT},  if a sequence of Leray-Hopf solutions $u_n$ satisfies 
(i) $u_n\in {\rm Lip}(0, T;  H^{-L}_{loc}(\Omega))$ for some large $L>0$, and 
(ii) the sequence $\omega_n=\curl u_n$ is precompact in $C(0, T; H^{-1}_{loc}(\Omega))$,
 then along a subsequence $u_n$ converges to a weak Euler solution $u$  strongly in $L^2(0, T; L^2_{loc}(\Omega))$. The assumptions of our Theorem \ref{theorem} of an $L^2_t$ structure-function bound \eqref{StuctFunc} does not allow us to conclude the latter condition (ii). It was also remarked in \cite{LNT} that it is not clear how to weaken the aforementioned assumption on $\om_n$ to $L^2(0, T; H^{-1}_{loc}(\Omega))$, which our assumptions do grant us.  In contrast, from our Lemma \ref{lemma3}, under the weaker precompactness of $\om_n$ in $C(0, T; H^{-1-\delta}_{loc}(\Omega))$  we obtained the convergence $u_n(t)\to u(t)$ in $L^2(\Omega)$ for almost every $t$. This combined with condition  \eqref{StuctFunc} implies that the limiting function $u$ is a weak Euler solution, by virtue of Constantin \& Vicol's argument, Lemma \ref{lemma4} herein. }
 
\red{We remark finally that if we had assumed the stronger condition that the sequence $u_n$ is uniformly bounded in $L^\infty(0,T;B_2^{\zeta_2/2,\infty}(\Omega))$, then from the precompactness established in our Lemma \ref{lemma2} and interpolation, condition (ii) would hold.  In this case, we could apply Theorem 1.1 of \cite{LNT} as a black box to obtain strong $L_t^2 L^2_{x,loc}$ convergence of the velocity (as a matter of fact, the proof given in \cite{LNT} is only for the case of $\Omega=\mathbb{R}^d$, but the proof works locally on domains since the argument is purely local).  We thank the anonymous referee for this remark. }
\end{rem}

 \section{Proof of Proposition \ref{prop}}
 
 \begin{proof}
 The proof consists of two steps, establishing both implications.
 \vspace{2mm}
 
\noindent \textbf{Step 1:  \eqref{StuctFunc} implies \eqref{vortbd}.}  Let $V\Subset U\Subset \Omega$. By Lemma \ref{lemma1} and the embedding $B^{\zeta_2/2, \infty}_2(U)\subset H^{\zeta_2/2-\ve}(V)$, $\ve\in (0, \zeta_2/2)$ (see Lemma \ref{embed}) we have $u_n$ is uniformly bounded in $H^{\zeta_2/2-\ve}(V)$. Consequently $\om_n$ is uniformly bounded in $H^{-1+\zeta_2/2-\ve}(V)$, proving \eqref{vortbd} with $\delta=\zeta_2/2-\ve>0$.
\vspace{2mm}

\noindent \textbf{Step 2:  \eqref{vortbd} implies  \eqref{StuctFunc}.} 
Let $V\Subset \Omega$ and pick two open sets $U\Supset W\Supset V$ and a scalar cut-off $\chi\in C^\infty_c(U)$ satisfying $\chi=1$ on $W$. Using the fact that $-\Delta u_n=\curl( \om_n)$ in 3D (in 2D, $-\Delta u_n=\nabla^\perp \omega_n$) we have
\begin{align}\nonumber
-\Delta(\chi u_n)&=\chi \curl (\om_n)-2\nabla \chi\cdot \nabla u_n-u_n\Delta\chi\\ \label{ellipticProp} 
&=\curl (\chi \om_n)-\nabla \chi \times \om_n -2\div (\nabla \chi \otimes u_n) + u_n\Delta\chi
\end{align}
where $\times$ denotes the vector cross product. Regard $\chi u_n$, $\chi \om_n$, $\om_n\times \nabla \chi$, $u_n\otimes \nabla \chi$ and $u_n\Delta\chi$ as vector fields defined over $\Rr^d$ taking value $0$ outside $U$. Note that $u_n\in H^1_0(\Omega)$ and $\om_n\in L^2(\Omega)$. Standard elliptic estimates for the problem \eqref{ellipticProp} in $\Rr^d$ yield
\begin{align*}
\Vert \chi u_n\Vert_{L^2(0, T; H^{\delta}(\Rr^d))}&\le C \Vert \curl (\chi \om_n)-\nabla \chi \times \om_n-2\div(\nabla \chi \otimes u_n)+u_n\Delta \chi\Vert_{L^2(0, T;  H^{-2+\delta}(\Rr^d))}\\
&\le C\| \chi \om_n\|_{L^2(0, T;  H^{-1+\delta}(\Rr^d))}+C\|\om_n\times \nabla \chi\|_{L^2(0, T;  H^{-2+\delta}(\Rr^d))}\\
&\qquad +C\|2\nabla \chi \otimes u_n+u_n\Delta\chi\Vert_{L^2(0, T;  L^2(\Rr^d))}\\
&\le C \| \chi \om_n\|_{L^2(0, T;  H^{-1+\delta}(\Rr^d))}+ C\|\om_n\times \nabla \chi\|_{L^2(0, T;  H^{-1+\delta}(\Rr^d))}+C\| u_n\|_{L^2(0, T; L^2(U))}.
\end{align*}
Since $\om_n\in L^2(\Omega)$ we have
\begin{align*}
\Vert \chi \om_n\Vert_{L^2(0, T;  H^{-1+\delta}(\Rr^d))}
&=\sup_{\| g\|_{H^{1-\delta}(\Rr^d)}\le 1}\int_{\Rr^d} \chi \om_n\cdot gdx =\sup_{\| g\|_{H^{1-\delta}(\Rr^d)}\le 1}\int_{U} \chi \om_n\cdot gdx.
\end{align*}
Note that for $\mu\in [0, 1]\setminus \{\frac 12\}$ the bilinear form $(g, h)\mapsto \int_{U} gh$ defines a continuous form on $H^{-\mu}(U)\times H^\mu_0(U)$. Without loss of generality, we assume $\delta\in(0,1/2)$.  Then
\begin{align*}
\Vert \chi \om_n\Vert_{L^2(0, T;  H^{-1+\delta}(\Rr^d))}&\le C_\chi \sup_{\| g\|_{H^{1-\delta}(\Rr^d)}\le 1}\| \om_n\|_{H^{-1+\delta}(U)}\| g\|_ {H^{1-\delta}(U)} \le C_\chi \| \om_n\|_{H^{-1+\delta}(U)}
\end{align*}
where we used the fact that $\| g\|_ {H^{1-\delta}(U)}\le \| g\|_ {H^{1-\delta}(\Rr^d)}$ since $1-\delta>0$. 
An analogous argument shows
\begin{align*}
\|\om_n\times \nabla \chi\|_{L^2(0, T;  H^{-1+\delta}(\Rr^d))}&\le C_\chi \| \om_n\|_{H^{-1+\delta}(U)}.
\end{align*}
Thus, combining the above estimates we have
\[
\Vert \chi u_n\Vert_{L^2(0, T; H^{\delta}(\Rr^d))}\lesssim \| \om_n\|_{H^{-1+\delta}(U)}+\| u_n\|_{L^2(0, T; L^2(U))}
\]
and $u_n$ is uniformly bounded in $H^\delta(W)\subset B^{\delta, \infty}_2(V)$  (see Lemma \ref{embed}). This concludes \eqref{StuctFunc} with $\zeta_2=2\delta$. 
\end{proof}

 \vspace{-2.5mm}
 \section{Discussion}
 
Under the assumption of uniformly-in-viscosity positive local scaling exponents for the second-order structure functions, we have established the weak inviscid limit on arbitrary domains with smooth boundary.  The imposed condition is no stronger than what is required on domains without boundaries \cite{Glimm,Isett17,DE17}  and is consistent with available data of turbulent flow on domains both with and without solid walls. As computing power improves and more experiments are run, the validity of \eqref{StuctFunc} can continue to be checked over longer scale ranges and higher Reynolds numbers.
In this way our results offer, in principle, experimentally refutable statements which can be used to assess the validity of the weak inviscid limit.

 The emergence of weak Euler solutions in the inviscid limit is closely related to anomalous dissipation of energy and Onsager's conjecture.  Anomalous dissipation is the phenomenon of non-vanishing dissipation of energy in the limit of zero viscosity
 \be\label{anom}
\lim_{\nu\to 0} \nu \int_0^T\int_\Omega |\nabla u^\nu(x,t)|^2dxdt >0.
 \ee
There is a vast amount of corroborator evidence for \eqref{anom}.  For numerical simulations of fluid flow on the torus, see e.g. \cite{KRS98,KIYIU03}.  The situation on domains with boundary is similar. For example, see data plotted in \cite{KRS84,PKW02} from wind turbulence experiments, in \cite{C97} more complex geometries and  in 2D numerical simulation with critical Navier-friction conditions \cite{Farge11} (see \cite{DN18} for a more extended discussion of these results).

 The vanishing or not  of the viscous dissipation is closely related to the issue of the strong inviscid limit and the appearance singularities in the solutions.    Clearly, sustaining dissipation in the inviscid limit \eqref{anom} requires divergence of velocity gradients as $\nu\to 0$.    In situations exhibiting anomalous dissipation, convergence to a strong Euler solution cannot hold \cite{K72}.   Onsager conjectured  \cite{O49} that \eqref{anom} could be captured by convergence of $u^\nu$ to a weak solution of the Euler equations which dissipates energy.  Further, he recognized the requirement of a more refined-type of singularity.  In particular, he asserted that dissipative weak Euler solutions must necessarily be less regular than 1/3--H\"{o}lder continuous.    
 
 On domains without boundary, Onsager's assertion has since been proved \cite{GLE94,CET94} and dissipative Euler solutions with less regularity have been constructed \cite{I16,BLSV17}.  See the recent review \cite{E18}.  On domains with boundary, there has been recent work deriving Onsager-type criteria for energy conservation \cite{DN18,BT18,BTW18}, and some constructions of non-conservative Euler solutions exist \cite{W17,B14}.  Thus far, none of the convex integration constructions of dissipative weak Euler solutions in domains either with or without boundary are known to arise as zero viscosity limits of Leray--Hopf weak solutions. \red{In fact, the importance of the Leray-Hopf class is highlighted by the fact that recently weak Euler solutions with arbitrary (in particular, non-monotonic) energy profiles have been constructed as limits of so-called ``very weak" solutions of Navier-Stokes which do not satisfy the energy inequality \cite{BV17}.} A major difference between domains with and without boundary from the point of view of the inviscid limit is the formation of thin, singular boundary layers which may be shed from the solid walls and become a source of anomalous dissipation in the bulk \cite{DN18}. Further understanding these issues remains an active and important area of research.

\appendix
\section{Besov embeddings}
 \begin{lemma}\label{embed}
 Let $U$ and $V$ be two open sets in $\Rr^d$, $d\ge 2$, such that $V\Subset U$. Let $s\in (0, 1)$ and $p\in [1, \infty)$. Then  the following embeddings are continuous
 \begin{itemize}
 \item  $B^{s, \infty}_p(U)\subset W^{s-\ve, p}(V)$ for all $\ve\in (0, s)$,
 \item  $B^{s, \infty}_p(U)\subset L^q(V)$ for all $q\in [1, \frac{dp}{d-sp})$,
 \item  $W^{s, p}(U)\subset B^{s, \infty}_p(V)$. 
 \end{itemize}
 \end{lemma}
\begin{proof}
Fix $\chi\in C^\infty_c(U)$ a cut-off function satisfying $\chi\equiv 1$ on $V$. Consider $u\in B^{s, \infty}_p(U)$ defined by the norm \eqref{besov-def}. We will show that $\chi u\in B^{s, \infty}_p(\Rr^d)$. Indeed, for any $r$ in $\Rr^d$ we have
\begin{align*}
\frac{1}{|r|^{ps}}{\| \chi u(\cdot+r)-\chi u(\cdot)\|_{L^p(\Rr^d)}^p}
&=\frac{1}{|r|^{ps}}\int_{x\in U,~x+r\notin U}|\chi u(x)|^pdx+\frac{1}{|r|^{ps}}\int_{x+r\in U,~x\notin U}|\chi u(x+r)|^pdx\\
&\qquad+\frac{1}{|r|^{ps}}\int_{x,~x+r\in U}|\chi u(x+r)-\chi u(x)|^pdx =: A +B +C.
\end{align*}
$A$ can be bounded by
\[
|A|= \frac{1}{|r|^{ps}}\int_{x\in U,~x+r\notin U}|\chi(x)-\chi(x+r)|^p |u(x)|^pdx \le \| \chi\|_{C^{s}(\Rr^d)}^p\| u\|_{L^p(U)}^p
\]
and similarly for $B$. Regarding $C$ we use the Besov regularity of $u$ to have
\begin{align*}
|C|&\le \frac{1}{|r|^{ps}}\int_{x,~x+r\in U}|\chi(x+r)|^p| u(x+r)- u(x)|^pdx+ \frac{1}{|r|^{ps}}\int_{x,~x+r\in U}|\chi(x+r)-\chi(x)|^p|u(x)|^pdx\\
&\le \| \chi\|_{L^\infty(\Rr^d)}^p\| u\|^p_{B^{s,\infty}_p(U)}+\| \chi\|_{C^{s}(\Rr^d)}^p\| u\|_{L^p(U)}^p.
\end{align*}
Putting together the above estimates we obtain
\bq\label{Bineq}
\| \chi u\|_{B^{s,\infty}_p(\Rr^d)}\le C\| \chi\|_{C^s(\Rr^d)}\| u\|_{B^{s,\infty}_p(U)}.
\eq
Using this together with the embeddings $B^{s,\infty}_p(\Rr^d)\subset W^{s-\ve, p}(\Rr^d)\subset L^{\frac{dp}{d-(s-\ve)p}}(\Rr^d)$, $\ve\in (0, s)$,  and the fact that $\chi=1$ on $V$ we conclude $u\in W^{s-\ve, p}(V)$ and $u\in L^q(V)$ for all $q<\frac{dp}{d-sp}$.

Now let $u\in W^{s, p}(U)$ with $p\in [1, \infty)$ and $s\in (0, 1)$. With $\chi\in C^\infty_c(U)$, $\chi\equiv 1$ on $V$ one can prove as in \eqref{Bineq} that $\chi u\in W^{s, p}(\Rr^d)=B^{s, p}_p(\Rr^d)\subset B^{s, \infty}_p(\Rr^d)$. Therefore, $u\in B^{s, \infty}_p(V)$. 
\end{proof}

 \subsection*{Acknowledgments}  We are grateful to Peter Constantin and Vlad Vicol for useful discussions, \red{as well as the anonymous referee for comments, all of  which greatly improved the paper}.  Research of TD is partially supported by NSF-DMS grant 1703997.


\begin{thebibliography}{99}

\bibitem{JM17}
Jiang, N., and  Masmoudi, N.: Boundary Layers and Incompressible Navier-Stokes-Fourier Limit of the Boltzmann Equation in Bounded Domain I. Comm. Pure. Appl. Math. 70.1 (2017): 90-171.
\bibitem{L34}
Leray, J.: Sur le mouvement d'un liquide visqueux emplissant l'espace, Acta Math. {\bf 63}: 193--248, (1934)
\bibitem{IS11}
Iftimie, D., and  Sueur, F.: Viscous boundary layers for the Navier-Stokes equations with the Navier slip conditions. Arch. Ration. Mech. Anal. 199.1: 145-175,  (2011).
\bibitem{K72}
Kato, T.: Nonstationary flows of viscous and ideal fluids in $\mathbb{R}^3$. J. Func. Anal. 9.3 (1972): 296-305.
\bibitem{P86}
Constantin, P.: Note on loss of regularity for solutions of the 3D incompressible Euler and related equations. Comm. Math. Phys. 104.2 (1986): 311-326.
\bibitem{kato}
Kato, T.:
\newblock Remarks on zero viscosity limit for nonstationary {N}avier-{S}tokes
  flows with boundary.
\newblock In {\em Seminar on nonlinear partial differential equations
  ({B}erkeley, {C}alif., 1983)}, volume~2 of {\em Math. Sci. Res. Inst. Publ.},
  pages 85--98. Springer, 1984.
\bibitem{BardosTiti13}
Bardos, C. and Titi, E.:
\newblock Mathematics and turbulence: where do we stand?
\newblock {\em Journal of Turbulence}, 14(3):42--76, 2013.
\bibitem{ceiv}Constantin, P., Elgindi, T., Ignatova, M. , and Vicol, V.: Remarks on the inviscid limit for the Navier-Stokes equations
for uniformly bounded velocity fields. SIAM J. Math. Anal., to appear (2017).
\bibitem{ConstantinKukavicaVicol15}
Constantin, P., Kukavica, I., and Vicol, V.:
\newblock On the inviscid limit of the {N}avier-{S}tokes equations.
\newblock {\em Proc. Amer. Math. Soc.}, 143(7):3075--3090, 2015.
\bibitem{Kelliher07}
 Kelliher, J.:
\newblock On {K}ato's conditions for vanishing viscosity.
\newblock {\em Indiana Univ. Math. J.}, 56(4):1711--1721, 2007.
\bibitem{TemamWang97b}
Temam, R. and Wang, X.:
\newblock On the behavior of the solutions of the {N}avier-{S}tokes equations
  at vanishing viscosity.
\newblock {\em Ann. Scuola Norm. Sup. Pisa Cl. Sci. (4)}, 25(3-4):807--828
  (1998), 1997.
\newblock Dedicated to Ennio De Giorgi.
\bibitem{Kelliher06}
Kelliher, J.: Navier--Stokes equations with Navier boundary conditions for a bounded domain. 38.1 (2006): 210-232.
\bibitem{SammartinoCaflisch98b}
Sammartino, M., and Caflisch, R.E.:
\newblock Zero viscosity limit for analytic solutions of the {N}avier-{S}tokes
  equation on a half-space. {II}. {C}onstruction of the {N}avier-{S}tokes
  solution.
\newblock {\em Comm. Math. Phys.}, 192(2):463--491, 1998.
\bibitem{GV16}
G\'{e}rard-Varet, D.,  Maekawa, Y., and  Masmoudi, N.: Gevrey stability of Prandtl expansions for 2D Navier-Stokes." arXiv preprint arXiv:1607.06434 (2016)
\bibitem{Maekawa14}
Maekawa, Y.:
\newblock
On the inviscid limit problem of the vorticity equations for viscous incompressible flows in the half-plane.
\newblock
{\em Comm. Pure Appl. Math.}, 67(7):1045--1128, 2014.
\bibitem{Kelliher09}
 Kelliher, J.:
\newblock On the vanishing viscosity limit in a disk.
\newblock {\em Math. Ann.}, 343(3):701--726, 2009.
\bibitem{LopesMazzucatoLopesTaylor08}
Lopes~Filho, M., Mazzucato, A., Nussenzveig~Lopes, H., and Taylor, M.:
\newblock Vanishing viscosity limits and boundary layers for circularly
  symmetric 2{D} flows.
\newblock {\em Bull. Braz. Math. Soc. (N.S.)}, 39(4):471--513, 2008.
\bibitem{MazzucatoTaylor08}
Mazzucato, A.~ and Taylor, M.:
\newblock Vanishing viscosity plane parallel channel flow and related singular
  perturbation problems.
\newblock {\em Anal. PDE}, 1(1):35--93, 2008.
\bibitem{DM87}
DiPerna, R. J., and Majda, A. J. (1987). Oscillations and concentrations in weak solutions of the incompressible fluid equations. Comm.  Math. Phys., 108(4), 667-689.
\bibitem{Glimm}
Chen, G, and  Glimm, J.: Kolmogorov's Theory of Turbulence and Inviscid Limit of the Navier-Stokes Equations in  ${\mathbb {R}^ 3} $. Comm. Math. Phys. 310.1 (2012): 267-283.
\bibitem{Isett17}
Isett, P.: Nonuniqueness and existence of continuous, globally dissipative Euler flows.  preprint arXiv:1710.11186 (2017).
\bibitem{DE17}
Drivas, T.D., and  Eyink, G.L.: An Onsager singularity theorem for Leray solutions of incompressible Navier-Stokes. preprint arXiv:1710.05205 (2017).
\bibitem{Filho05} 
Filho, MC Lopes, Nussenzveig Lopes, HJ, and  Planas, G.: On the inviscid limit for two-dimensional incompressible flow with Navier friction condition. SIAM J. Math. Anal.36.4 (2005): 1130-1141.
\bibitem{CV17}
Constantin, P., and Vicol, V.: Remarks on high Reynolds numbers hydrodynamics and the inviscid limit. J. Nonlin. Sci. 28.2 (2018): 711-724.
\bibitem{W17}
Wiedemann, E.: Weak-strong uniqueness in fluid dynamics. preprint arXiv:1705.04220 (2017).
\bibitem{GR12}
Girault, V., and Raviart, P.. Finite element methods for Navier-Stokes equations: theory and algorithms. Vol. 5. Springer Science \& Business Media, 2012.
\bibitem{LNT}
Lopes Filho, M. C., Nussenzveig Lopes, H. J., and Tadmor, E.: Approximate solutions of the incompressible Euler equations with no concentrations.  Annales de l'Institut Henri Poincare (C) Nonlinear Analysis (Vol. 17, No. 3, pp. 371-412), 2000.
\bibitem{KRS98}
Sreenivasan, K.~R.: An update on the energy dissipation rate in isotropic turbulence, 
Phys. Fluids {\bf 10}: 528--529,  (1998)
\bibitem{KIYIU03}
Kaneda, Y., Ishihara, T., Yokokawa, M., Itakura, K.,  and Uno, A.: 
Energy dissipation rate and energy spectrum in high resolution direct numerical simulations 
of turbulence in a periodic box, Phys. Fluids {\bf 15}: L21--L24,  (2003)
\bibitem{KRS84} 
Sreenivasan, K.~R.: On the scaling of the turbulence energy dissipation rate, 
Phys. Fluids {\bf 27}: 1048--1051,  (1984)
\bibitem{PKW02}
 Pearson, B., Krogstad, P., and van de Water, W.:
Measurements of the turbulent energy dissipation rate, Phys. Fluids {\bf 14}: 1288--1290,  (2002)
\bibitem{C97}
Cadot, O., Couder, Y., Daerr, A., Douady, S., and Tsinober, A.: Energy injection in closed turbulent flows: Stirring through boundary layers versus inertial stirring. Phys. Rev. E, 56(1), 427, (1997).
\bibitem{Farge11}
R. Nguyen van yen, M. Farge, and K. Schneider.: Energy dissipating structures produced by walls in two-dimensional flows at vanishing viscosity. Phys. Rev. Lett. 106.18 (2011): 184502.
\bibitem{DN18}
Drivas, T. D., and Nguyen, H. Q.: Onsager's conjecture and anomalous dissipation on domains with boundary. SIAM J. Math. Anal., 50(5), 4785–4811 (2018). 
\bibitem{O49} Onsager, L.: Statistical hydrodynamics. Il Nuovo Cimento (Supplemento), {\bf{6}}: 279-287 (1949)
\bibitem{GLE94} 
Eyink, G.~L.: Energy dissipation without viscosity in ideal hydrodynamics I. Fourier analysis and local energy transfer, Physica D {\bf 78}: 222--240,  (1994)
\bibitem{CET94}Constantin, P. , W. E, and Titi, E.: Onsager's conjecture on the energy conservation for solutions of Euler's equation. Comm. Math. Phys. {\bf{165}}: 207-209,  (1994)
\bibitem{I16} 
Isett, P.: A proof of Onsager's conjecture, preprint arXiv:1608.08301,  (2016)
\bibitem{BLSV17}
Buckmaster, T., De Lellis, C., Sz\'ekelyhidi Jr., L. and Vicol, V.: 
Onsager's conjecture for admissible weak solutions, 
preprint arXiv:1701.08678,  (2017)
\bibitem{E18}
Eyink, G.L.: Review of the Onsager ``Ideal Turbulence" Theory. preprint arXiv:1803.02223 (2018).
\bibitem{BT18}
Bardos, C., and Titi, E.:  Onsager's Conjecture for the Incompressible Euler Equations in Bounded Domains. Arch. Ration. Mech. Anal.228.1 (2018): 197-207.
\bibitem{BTW18}
Bardos, C., Titi, E., Wiedemann, E.: Onsager's Conjecture with Physical Boundaries and an
Application to the Viscosity Limit, preprint arXiv:1803.04939, (2018)
\bibitem{B14}
Bardos, C., Sz\'ekelyhidi Jr, L. and  Wiedemann, E.: Non-uniqueness for the Euler equations: the effect of the boundary.  Russian Mathematical Surveys 69.2 (2014): 189.
\bibitem{BV17}
Buckmaster, T., and Vicol, V.: Nonuniqueness of weak solutions to the Navier-Stokes equation. arXiv preprint arXiv:1709.10033 (2017).
\end{thebibliography}
\end{document}